\documentclass{amsart}

\usepackage[utf8]{inputenc}
\usepackage{graphicx}
\usepackage{amsthm}
\usepackage{amssymb}
\usepackage{blindtext}
\usepackage{amsmath}
\usepackage[inline]{enumitem}

\begin{document}

\theoremstyle{plain}
\newtheorem{thm}{Theorem}[section]
\newtheorem{Propos}[thm]{Proposition}
\newtheorem{lem}[thm]{Lemma}
\newtheorem*{lemanonum41}{Lemma 5.3}
\newtheorem*{lemanonum53}{Lemma 5.3 in [2]}
\newtheorem*{lemanonum}{Lemma}
\newtheorem*{definonum}{Definition}
\newtheorem*{teononum}{Theorem}
\theoremstyle{definition}
\newtheorem{defi}{Definition}
\newtheorem*{ams}{AMS subject classification}
\theoremstyle{remark}
\newtheorem*{obs}{Remark}
\newtheorem*{ack}{Acknowledgments}
\newtheorem*{nota}{Note}
\theoremstyle{definition}
\newtheorem*{demsigma}{Proof of Lemma \ref{sigma}}
\newtheorem*{demTeo}{Proof of Theorem \ref{teoesf}}
\newtheorem*{demnon}{Proof of Theorem \ref{nonexis}}

\newcommand{\al}{{\alpha}}
\newcommand{\ti}{{\theta}}
\newcommand{\Ti}{{\Theta}}
\newcommand{\ep}{{\epsilon}}

\markboth{CAROLINA A. REY}
{ELLIPTIC EQUATIONS WITH CRITICAL EXPONENT ON A INVARIANT REGION OF $\mathbb{S}^3$}
%

\title{ELLIPTIC EQUATIONS WITH CRITICAL EXPONENT ON \\ A TORUS INVARIANT REGION OF $\mathbb{S}^3$}

\author{{CAROLINA A. REY}}

\address{Departamento de Matem\'atica, Universidad de Buenos Aires\\
Ciudad Universitaria, Pabell\'on I, (C1428EGA), Buenos Aires, Argentina\\
carey@dm.uba.ar}

\begin{abstract}
We study the multiplicity of positive solutions of the critical elliptic equation:
\[
 \Delta_{\mathbb{S}^3} U = -(U^5 +\lambda U)  \hspace{0.3cm}\hbox{ on } \Omega
\]
that vanish on the boundary of $\Omega$, where $\Omega$ is  a region of $\mathbb{S}^3$ which is invariant by the natural $\mathbb{T}^2$-action.
H. Brezis and L. A. Peletier in \cite{BP} consider the case in which $\Omega$ is invariant by the $SO(3)$-action, namely, when $\Omega$ is a spherical cap.
We show that the number of solutions increases as $\lambda \to -\infty$, giving an answer of a particular case of an open problem proposed by 
H. Brezis and L. A. Peletier in \cite{BP}.
\end{abstract}

\maketitle


\section{Introduction}

We consider the critical elliptic equation:

\begin{equation}
\label{laplac}
 \Delta_{\mathbb{S}^3} U = -\left(U^5 +\lambda U\right)  \hspace{0.5cm}\hbox{on } \Omega 
\end{equation}
where $ \Delta_{\mathbb{S}^3}$ is the Laplace-Beltrami operator on $\mathbb{S}^3$ and $\Omega$ is a particular open subset of $\mathbb{S}^3$.
We look for positive solutions of (\ref{laplac}) such that

\begin{equation}
\label{borde}
 U=0 \hspace{0.5cm}\hbox{on } \partial{\Omega}.
\end{equation}
Problems of this kind have attracted the attention of several researchers with the aim to understand the existence and 
properties of the solutions. 

H. Brezis and L. Nirenberg considered the problem in $\mathbb{R}^3$:
\begin{equation}
\label{BN}
 \Delta_{\mathbb{R}^3} U = -\left(U^5 +\lambda U\right), U>0 \hbox{ in } B_{R^*}, \hspace{0.2cm} U=0 \hspace{0.2cm}\hbox{ on } \partial{B_{R^*}}
 \end{equation}
where $B_{R^*}$ is the ball of radius $R^*$ of $\mathbb{R}^3$. Using variational techniques, they obtained in \cite{BN} 
necessary and sufficient conditions on the value of $\lambda$ for the existence of a solution. 
This solution was shown to be unique by M. K. Kwong and Y. Li in \cite{KL}. This is now called the 
Brezis-Nirenberg problem and there are numerous results about solutions of this problem in
different open subsets of $\mathbb{R}^n$.

The case when Euclidean space is replaced by $\mathbb{S}^3$ was considered in \cite{BB}, \cite{BnP}, \cite{BPP} and \cite{BP}.
Let $D_{\ti^*}$ be a geodesic ball in the $3$-dimensional sphere centered at the North pole with geodesic radius $\ti^*$.
Problem (\ref{laplac})-(\ref{borde}) with  $\Omega=D_{\ti^*}$ has been investigated by C. Bandle and R. Benguria in \cite{BB}, C. Bandle and L.A. Peletier 
in \cite{BnP} and H. Brezis and L. A. Peletier in \cite{BP}  in order to identify the range of values of the parameters $\ti^*$ and $\lambda$ 
for which there exists a solution. It is well-known that 
the method of moving planes can be applied  when $\ti^*<\pi/2$ (which means that the geodesic ball is contained in
a hemisphere) to prove that all solutions are radial (see for instance \cite{P} and \cite{KP}). The value $\lambda =-3/4$ is special 
since $\Delta_{\mathbb{S}^3} -3/4$ is the conformal Laplacian on $\mathbb{S}^3$ and Eq. (\ref{laplac}) is then the Yamabe equation: 
in this case it is known that there are no nontrivial solutions satisfying (\ref{borde}).  The cases $\lambda >-3/4$ and $\lambda <-3/4$ present
very different features. We will be interested in the second case. In particular, the situation when $\lambda \to -\infty$ studied by H. Brezis and L. A. Peletier in \cite{BP}. 
The main result in  \cite{BP} reads:
\par

\begin{teononum}[{\bf H. Brezis and L. A. Peletier}]
Given any $\ti^* \in (\pi/2,\pi)$ and any $k\geq 1$, there exists a constant $A_k > 0$ such that for $\lambda < -A_k$, 
problem (\ref{laplac})-(\ref{borde})  with  $\Omega=D_{\ti^*}$ has at least $2k$ positive radial solutions such that
$U(\hbox{North pole}) \in (0, |\lambda|^{1/4})$.
\end{teononum}

This result was extended by C. Bandle and J. Wei in \cite{BnW, BnWei} to general dimensions and non-critical
exponents. Also when
$\ti^* > \pi/2$ the moving plan method does not work and in
\cite{BnW} the authors establish the existence of positive nonradial solutions. In  \cite{BnWei} the authors  proved for balls of geodesic
radius $\ti^* > \pi/2$ the existence of radially symmetric clustered layer solutions as $\lambda \to -\infty$.

Inspired by the  theorem of H. Brezis and L. A. Peletier, we study problem (\ref{laplac})-(\ref{borde}) for the special case where $\Omega$ is a 
torus invariant region of $\mathbb{S}^3$. 
The spherical caps $D_{\ti^*}$ are invariant by the codimension one action of $O(3)$ on $\mathbb{S}^3$. The poles
are the singular orbits of the action and the spherical caps are the geodesic tubes around one of the singular orbits. 
In this paper we will consider the torus action on $\mathbb{S}^3$, which is the other codimension one isometric action. Consider 
$
\mathbb{T}^2 = \mathbb{S}^1 \times \mathbb{S}^1
$
and the natural action $\mathbb{T}^2 \times \mathbb{S}^3 \to \mathbb{S}^3$ given by

\begin{equation}
\label{action}
 (\alpha, \beta)(x, y, z, w)=(\alpha \cdot (x,y), \beta \cdot (z,w))
\end{equation}
where $\cdot$ is the complex multiplication. This  is an isometric, codimension one, action on $\mathbb{S}^3$ and there are two special orbits:
$\mathbb{S}^1\times \{0\}$ and $ \{0\}\times\mathbb{S}^1$. The distance between these two singular orbits is $\pi/2$.
As in the case of spherical caps studied by Brezis and Peletier, we consider an open set $\Omega$ which is the geodesic tube 
around one of the singular orbits:
\[
 \Omega=\{ \tilde{x}\in\mathbb{S}^3/ \hbox{ dist} (\tilde{x},\mathbb{S}^1\times {0})\leq\ti_1\},
\]
with $\ti_1\in(0,\pi/2)$.

Now we present a change of variables leading to a different formulation of problem (\ref{laplac})-(\ref{borde}). 
With this aim, we introduce the next local coordinates in $\mathbb{R}^4$:
\begin{equation}
\label{coord}
\left\{
\begin{array}{rcl}
x_1&=&r\cos(\ti)\cos(\eta_1)\\
x_2&=&r\cos(\ti)\sin(\eta_1)\\
x_3&=&r\sin(\ti)\cos(\eta_2)\\
x_4&=&r\sin(\ti)\sin(\eta_2)\\
\end{array}
\right.
\end{equation}
where $r = \sqrt{x_1^2+x_2^2+x_3^2+x_4^2}$, $0 \leq \ti<\pi/2$, $0\leq \eta_1, \eta_2 \leq 2\pi$. 
In these coordinates, the unit sphere $\mathbb{S}^3$ can be parameterized by
$r = 1$, $\{ 0\leq \ti \leq \pi/2, 0< \eta_1, \eta_2 < 2\pi\}$.
The round metric $g$ on the 3-sphere in these coordinates is given by
\[
ds^2= d\ti^2+ \cos^2(\ti) d\eta_1^2+ \sin^2(\ti) d\eta_2^2
\]

Note that $\ti$ is the geodesic distance to the orbit $\mathbb{S}^1\times \{0\}$.
Then
$$
\Omega=\Omega_{\ti_1}=\{ \left(x_1,x_2,x_3,x_4\right)\in \mathbb{S}^3/ 0 \leq \ti \leq \ti_1\}
$$
with $\ti_1\in(0,\pi/2).$ Consequently $\Omega$ is an open subset in $\mathbb{S}^3$ invariant by the $\mathbb{T}^2$-action.
Recall that the Beltrami-Laplace operator on $\mathbb{S}^3$ in local coordinates is given by:
\begin{equation}
 \Delta_{\mathbb{S}^3}=\frac{1}{\sqrt{|g|}}\sum_{i=1}^3 
\frac{\partial}{\partial{\eta_i}}\left(g_{ii}^{-1}\sqrt{|g|}\frac{\partial}{\partial{\eta_i}}\right).
\end{equation}

Suppose that the function $U:\Omega \to \mathbb{R}$ is invariant by the $\mathbb{T}^2$-action. 
Then $U(x,y,z,w)=u(\ti)$ for some function $u:[0,\ti_1]\to \mathbb{R}$ and since  
$$|g|= \cos^2(\ti)  \sin^2(\ti),$$ 
the Laplace-Beltrami operator on $\mathbb{S}^3$ applied to $U$ takes the form:
\begin{equation}
\label{DeltaU}
\begin{array}{rcl}
\Delta_{\mathbb{S}^3} U 
&=&\displaystyle\frac{1}{\cos(\ti)  \sin(\ti)} \displaystyle\frac{d}{d\ti}\left(\cos(\ti)  \sin(\ti)\displaystyle\frac{d{u}}{d{\ti}}\right)\\
\\
&=&u''(\theta) +\displaystyle\left( \displaystyle\frac{ \cos(\theta)}{\sin(\theta)}- \displaystyle\frac{\sin(\theta)}{\cos(\theta)}  \right)  u'(\theta)\\
\\
 &=& u''(\theta) + 2 \displaystyle\frac{ \cos(2\theta)}{\sin(2\theta)} u'(\theta)\\
\end{array}
\end{equation}
It follows that if we restrict the  original problem to functions which are invariant by the $\mathbb{T}^2$-action,
it is equivalent to finding solutions of:

\begin{equation}
\label{toro}
\left\{
\begin{array}{rcl}
 u''(\theta) + 2 \frac{ \cos(2\theta)}{\sin(2\theta)} u'(\theta) &=&
 \lambda \left(u(\theta)^5 - u(\theta)\right), \hspace{0.5cm} u>0  \hspace{0.5cm}\hbox{on } (0,\theta_1)\\
u'(0)&=&0\\
u(\theta_1)&=&0\\
\end{array}
\right.
\end{equation}

We also are interested to study solutions of the equation invariant by the $\mathbb{T}^2$-action
in the whole sphere $\mathbb{S}^3$: 

\begin{equation}
\label{esfera}
 \Delta_{\mathbb{S}^3} U = \lambda \left(U^5 -U\right),  \hspace{0.5cm} U>0  \hspace{0.5cm} \hbox{on } \mathbb{S}^3\\
\end{equation}

Positive solutions of (\ref{esfera}) are called ``ground state" solutions. We have the following result analolgous to 
\cite[Theorem 1.6]{BP}:

\begin{thm}\label{teoesf} 
Let $n\geq 1$ and $\lambda \in [-(2n+2)(2n+3),-(2n)(2n+1))$. Then for every $k\in \{1,2, \dots, n\}$ 
there exists at least one solution $U_k$ of problem (\ref{esfera}), where $U_k=u_k(\theta)$ has the following propieties:
\begin{itemize}
 \item $u_k$ has exactly $k$ local maximum on $(0,\frac{\pi}{2})$
 \item $u_k(\pi/2-\theta)=u_k(\theta)$ for $\theta \in (0, \frac{\pi}{2})$;
 \item $u_k(0)<1$.
\end{itemize}
\end{thm}

We are interested in positive solutions of (\ref{toro}) with initial value in the interval $(0,1)$. 
Firstly we prove a theorem of nonexistence:
\begin{thm}
\label{nonexis}
If $\theta_1 \in (0, \pi/4)$, then there are no solutions of (\ref{toro}) with initial value in the interval $(0,1)$.
\end{thm}

This means that the solutions of (\ref{toro}) with initial value in the interval $(0,1)$ do not vanish before $\pi/4$. 
However, we shall prove the existence of an increasing number of solutions of problem (\ref{toro}) as $\lambda$ goes to $-\infty$ 
with initial value in the interval $(0,1)$, which gives a partial positive answer to the open problem 8.3 proposed by H. Brezis and 
L. A. Peletier in \cite{BP}. Our main result in this paper is the following

\begin{thm}
\label{ksoluciones}
Given any $k\geq 1$ and any $\theta_1 > \pi/4$, then there exists a constant $A_k > 0$ such that for $\lambda <-A_k$ 
problem  (\ref{toro}) has at least $2k$ solutions with initial value in the interval $(0,1)$.
\end{thm}
Our approach mainly relies upon a method that has been successfully used in \cite{BP}. First we use this method to show 
that there exists at least $2$ solutions of problem (\ref{toro}) with initial value in the interval $(0,1)$ that have a single spike or maximum.
The next step is to prove the theorem in the case $k=2$ using the same techniques. Finally the theorem follows by induction.

\begin{figure}[th]
    \centering
    \includegraphics[width=0.5\textwidth]{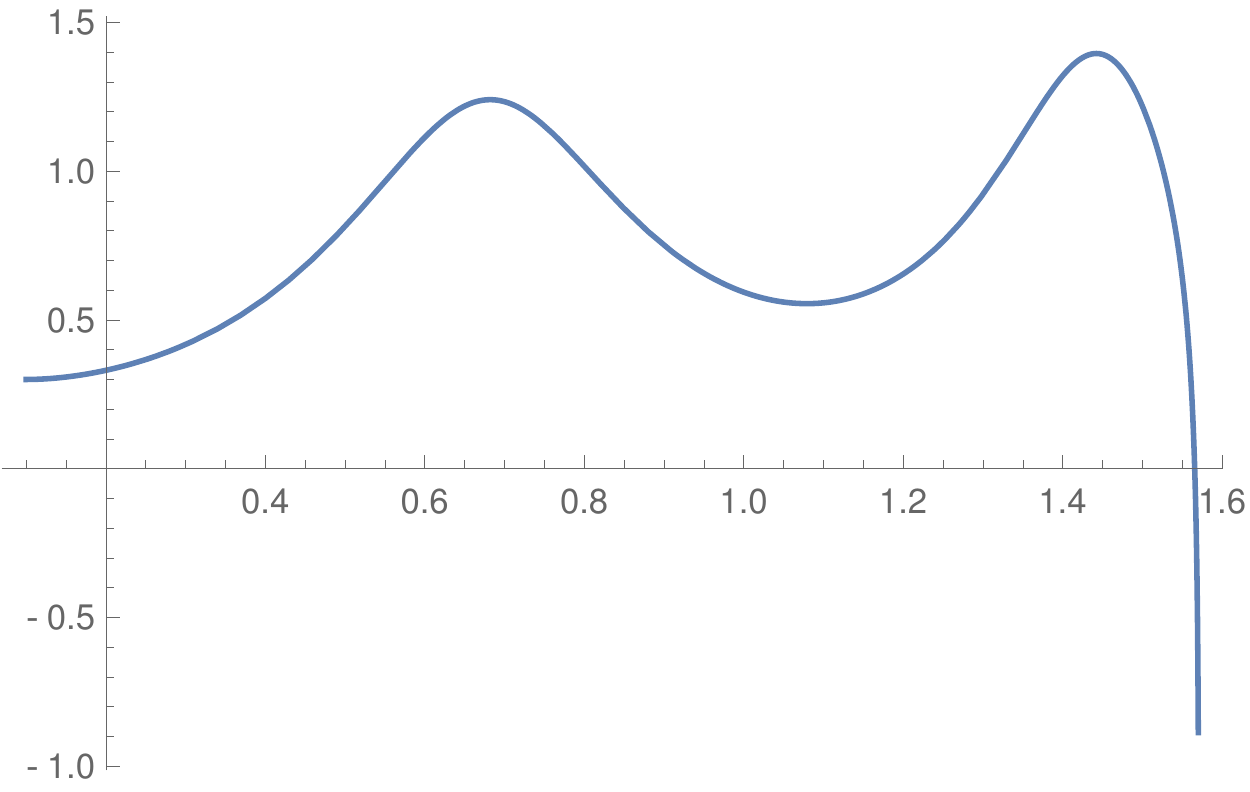}
\vspace*{8pt}
\caption{Two-spike solution $u$ of problem (\ref{toro}) for $\lambda=-25$ and $u(0)=0.3$}
    \label{fig:2-spike}
\end{figure}
 \begin{figure}[th]
    \centering
    \includegraphics[width=0.5\textwidth]{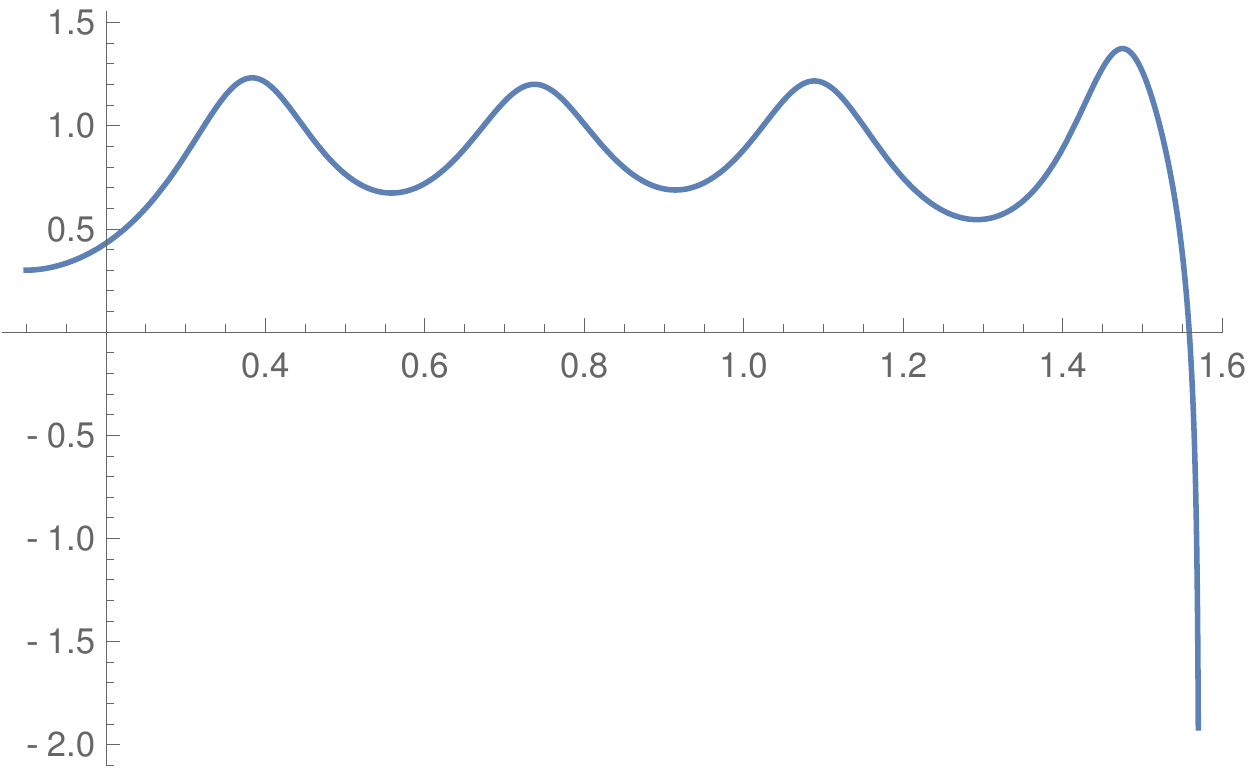}
    \vspace*{8pt}
    \caption{Four-spike solution $u$ of problem (\ref{toro}) for $\lambda=-100$ and $u(0)=0.3$}
    \label{fig:4-spike}
\end{figure}

The paper is organized as follows. 
In section $2$  we will study properties of the ground state solutions and prove Theorem \ref{teoesf}.
Section $3$ contains some results about auxiliary linear problems, that will help us to prove the main theorem in next section. 
Theorem \ref{nonexis} will be proved in section $4$, as well as Theorem \ref{ksoluciones}.


\section{Positive solutions on $\mathbb{S}^3$}

In this section we present a detailed study of the problem obtained by linearizing Eq. (\ref{toro}) around 
the nontrivial constant solution when $\lambda<0$. Then we use these results to prove Theorem \ref{teoesf}. 
Let $\alpha \in (0,1)$, $\lambda <0$,  and denote by $u_{\alpha , \lambda} (\ti )$ the solution of:

\begin{equation}
\left\{
\begin{array}{rcl}
 u''(\theta) + 2 \frac{ \cos(2\theta)}{\sin(2\theta)} u'(\theta) &=& \lambda \left(u(\theta)^5 - u(\theta)\right)\\
 u(0)&=&\al\\
 u'(0)&=&0\\
\end{array}
\right.
\label{esfera2}
\end{equation}

There is a constant solution $u_{1, \lambda} \equiv 1$, and it is important to understand the behavior of the solutions $u_{\al ,\lambda} (\ti)$ with $\al$ close to $1$.
With this aim, consider the function
\[
 w_\lambda(\theta)=\frac{d}{d\alpha} \Big|_{\al=1} u_{\al , \lambda }(\theta ). 
\]
Then $w_\lambda$ is the solution of the linear problem

\begin{equation}
\label{eqw}
\left\{
\begin{array}{rcl}
 w''(\theta) + 2 \frac{ \cos(2\theta)}{\sin(2\theta)} w'(\theta) &=&4\lambda w, \\
w(0)&=&1, \\
w'(0)&=&0.\\
\end{array}
\right.
\end{equation}

This is the eigenvalue equation for $ \Delta_{\mathbb{S}^3}$ restricted to functions invariant by the $\mathbb{T}^2$-action.
It can be understood for instance adapting  the techniques used by J. Petean in \cite{J} (for the case of radial functions). We will
sketch the proofs briefly for completeness.

Let $\lambda_n:=-n(n+1)$. 

If we denote by $F_c (\varphi ) (\theta ) =  {\varphi}'' (\theta )+ 2 \frac{ \cos(2\theta)}{\sin(2\theta)} {\varphi}'(\theta)  -c \varphi (\theta )$ then by 
a direct computation:

$$F_c (\cos^k (2\theta ))= (4\lambda_n - c ) \cos^k (2\theta ) +4k(k-1) \cos^{k-2} (2\theta ) .$$

\begin{lem}
The solution $w_{\lambda_n}$ of (\ref{eqw}) is a linear combination of $\cos^{n-2j}(2\theta)$, where $0 \leq 2j \leq n$.
\end{lem}

It then follows that $w_{\lambda_n} (\pi /2 ) =(-1)^n$. If $n$ is odd then $w_{\lambda_n} (\pi /4 )=0$ and if 
$n$ is even then $w_{\lambda_n}'  (\pi /4 )=0$.

It follows from Sturm-Liouville theory that the number of zeros of $w_{\lambda}$ in $(0,\pi /2)$  is a non-increasing function of 
$\lambda$ ($< 0$). It is then easy to see that:

\begin{lem}
The solution $w_{\lambda_n}$ has exactly $n$ zeros in the interval $(0, \frac{\pi}{2})$ and therefore exactly $n-1$ critical points in $(0, \frac{\pi}{2})$.
\end{lem}

And using again Sturm-Liouville theory and the previous comments it follows:

\begin{lem}
\label{extrloc}
If $\lambda \in [\lambda_{2n+2}, \lambda_{2n})$ then $w_{\lambda}$  has exactly $n$ critical points in the interval $(0, \pi/4)$.

\end{lem}

Denote by
\[
\tau_1^0(\lambda)<\tau_2^0(\lambda)< \dots <\tau_n^0 (\lambda)
\] the critical points of $w_\lambda$ in $(0, \pi /2 )$. 
Using the uniform continuity of the solution of problem (\ref{esfera2}) with respect to the initial value $\al$ we obtain:

\begin{lem}
\label{lim}
Suppose that $w_\lambda$ has a critical point $\tau_k^0(\lambda)$ for some $k\geq 1$. Then for $\al<1$ sufficiently close to $1$, 
the solution $u_{\al ,\lambda}$ has a critical point $\tau_k(\al)$ and
\[ \tau_k(\al) \to \tau_k^0(\lambda), \hbox{  as } \al \to 1.\]
\end{lem}

\begin{obs}
 $\tau_k=\tau_k(\al)$ is a continuous function (where it is defined)  and from (\ref{esfera2}) it is easy to see that
$u_{\al , \lambda}(\tau_j)>1$ if  $j$ is odd, and $u_{\al , \lambda}(\tau_j)<1$ if $j$ is even.
\end{obs}
\begin{lem}
\label{simetric}
If for any $\al \in (0,1)$ the solution $u_{\al}$ of problem (\ref{esfera2}) satisfies $u_{\al} '(\frac{\pi}{4})=0$, 
then $u_{\al} '(\frac{\pi}{2})=0$ and $u_{\al} (\ti)=u_{\al} (\frac{\pi}{2}-\ti)$ for $\ti \in [ \frac{\pi}{4},\frac{\pi}{2} ).$
\end{lem}

\begin{proof}
The function $v(\ti)=u_{\al}(\frac{\pi}{2}-\ti )$ for $\ti \in [\frac{\pi}{4},\frac{\pi}{2})$ is also a solution of the equation.
Moreover $v(\frac{\pi}{4}) = u_{\al}(\frac{\pi}{4})$ and $v'(\frac{\pi}{2}) = 0 = u_{\al} '(\frac{\pi}{4})$. 
Therefore $v = u_{\al}$ in $ [\frac{\pi}{4},\frac{\pi}{2})$ and the lemma follows.
\end{proof}
\hfill$\square$

\begin{lem}
\label{cercadecero}
If $\al$ is close to zero, then the solution $u_{\al}$ of problem (\ref{esfera2}) has
no local extremes on $(0, \frac{\pi}{4})$. 
\end{lem}

\begin{proof}
 For $\al$ close to $0$ the solution $u_\al$ increases slowly in interval $(0,\frac{\pi}{4})$ and stays less than $1$ in that interval.
Therefore it does not have any local extremes on $(0, \frac{\pi}{4})$.
\end{proof}
\hfill$\square$

\noindent Now define:
\begin{equation}
\label{F(u)}
 F(u):= \displaystyle\int_0^u \left(s^5-s\right)ds=  \frac{1}{6}u^6 - \frac{1}{2}u^2.
\end{equation}
Then $F(\alpha)<0$.
Note that $ F $ has only one positive zero $\sigma:= 3^{\frac{1}{4}}.$

\begin{lem}
\label{sigma}
If $\tau_j(\al)<\frac{\pi}{4}$, then
$ 0< u_{\al} (\tau_j(\al)) <\sigma. $
\end{lem}

To prove this lemma we consider the energy function defined by
\begin{equation}
\label{energy}
E_{\al}(\ti):=\frac{\left(u_{\al} '(\ti)\right)^2}{2} -\lambda
\left(\frac{\left(u_{\al}(\ti)\right)^6}{6}- \frac{\left(u_{\al}(\ti)\right)^2}{2}\right).
 \end{equation}
If $u_{\al}$ is a solution of problem (\ref{esfera2}) then we have 
\[
E_{\al} '(\ti)=-2 \frac{\cos(2\ti)}{\sin(2\ti)}u_{\al} '(\ti)^2. 
\]
Consequently $E_{\al} $ is decreasing on $[0,\frac{\pi}{4}]$ and
$
E_{\al}(0)= -\lambda F(\al).
$
\vspace{0.5cm}
\begin{demsigma}
Since $E_{\al} $ is decreasing on $[0,\frac{\pi}{4}]$ and $0<\tau_j(\al)\leq \frac{\pi}{4}$, it follows that
\[
E_{\al}(\tau_j) < E_{\al}(0)= -\lambda F(\al). 
\]
Consequently, since $E_\al(\tau_j(\al)) = -\lambda F(u_{\al}(\tau_j(\al)))$ and $0<\al<1$ we have that
\[F(u_{\al}(\tau_j(\al)))<  F(\al)<0. \]
This means that $ 0< u_{\al} (\tau_j(\al)) <\sigma, $ as asserted.
\end{demsigma}
\hfill$\square$
\vspace{0.5cm}

Next we define $\al_k^*$ as the infimum value of $\al$ for which $\tau_k(\al)$ exists on $(\al, 1)$:
\[
\al_k^*=\inf \{ \al_0\in(0,1): \hbox{ for } \al \in (\al_0,1) \hbox{ $u_\al$ has at least $k$ critical points on } (0,\pi/2)\}.
\]

\begin{lem}
\label{alfaestr}
Suppose that $\tau_k(\al)$ exists for some $\al<1$ sufficiently close to $1$ so that $\al_k^*$ is well
defined. Then there exists $\delta>0$ such that
\[
\tau_k (\al) \geq \frac{\pi}{4}  \hspace{0.5cm} \hbox{if } \al \in (\al_k^*, \al_k^*+ \delta)
\]
\end{lem}

\begin{proof}
If $\al_k^*=0$ then the assertion follows from Lemma 2.6.
Thus we may assume that $\al_k^* \in (0,1)$.
Suppose there exists a decreasing sequence $\{ \al_j \}$ such that
\[
\tau_k(\al_j) < \frac{\pi}{4}  \hspace{0.5cm} \hbox{and}  \hspace{0.5cm} \al_j \to \al_k^*.
\]
Since the sequences $\{ \tau_k(\al_j) \}$ and $\{ u_{\al_j}(\tau_k(\al_j)) \}$ are bounded by Lemma \ref{sigma}, it follows that there exists
$\tau_k^*\in [0,\frac{\pi}{4}]$ and an $u^* \in [0,\sigma]$ such that, taking a subsequence, we may soppose:
\[
\tau_k(\al_j) \to \tau_k^*   \hspace{0.5cm} \hbox{and}    \hspace{0.5cm} u_{\al_j}(\tau_k(\al_j)) \to u^*.
\]
If  $u^*$ is $1$ or $0$, then by uniqueness $u_{\al_k^*}$ is constant, which contradicts the fact that $\al_k^* \in (0,1)$.
If $u^*\in(0,1)$, then we use the Implicit Function Theorem with the function $G(\al, \ti)= u_{\al}'(\ti)$. Since $\al_k^* \neq 0,1$, it follows that $\frac{d}{d\ti}G(\al_k ^*, \ti) \neq 0$. 
But since $G(\al_k ^*, \ti)=0$, we have that $\tau_k(\al) $ is well defined for all $\al$ in a neighbourhood of $\al_k^*$, which contradicts the definition of $\al_k^*$.
\end{proof}
\hfill$\square$

We end this section with the proof of Theorem \ref{teoesf}. 
\begin{demTeo}
Suppose $n\geq 1$ and $\lambda \in [-(2n+2)(2n+3),-(2n)(2n+1) )$. 
Given  $k\in \{1,2, \dots, n  \}$ we will show that 
\[
\tau_k(\al_0)=\frac{\pi}{4}
\]
for some $\al_0$ and hence the solution $u_{\al_0}$ has $k$ local extremes on  
$(0, \frac{\pi}{4}]$.
Since $u'_{\al_0}(\pi /4 )=0$, by Lemma \ref{simetric} it follows that $u_{\al_0}$ 
satisfies $\rm (i)-\rm (iii)$ of Theorem \ref{teoesf}.
\par
By Lemmas 2.3 and  \ref{lim} , since $\lambda \in [\lambda_{2n+2}, \lambda_{2n})$ and $\al$ is close to $1$, the solution $u_{\al}$ 
has $n$ local extremes  $(0, \pi/4)$. Therefore $\tau_k(\al)<\frac{\pi}{4}$. 
On the other hand, by Lemma \ref{alfaestr} we know that if $\al$ is close to $\al_k^*$ then $\tau_k(\al)\geq \frac{\pi}{4}$. 
By continuous dependence it follows that there is $\al_0$ such that $\tau_k(\al_0)=\frac{\pi}{4}$. 

\end{demTeo}
\hfill$\square$


\section{Auxiliary results}
In this section we will establish three auxiliary results that we will need to prove our main Theorem in next section.

\begin{lem}
\label{lemaphi}
Let $\delta , \kappa>0$ and $K$ be constants. Then there are constants $\epsilon_1>0$ and $\beta>0$ such that the solution $\varphi_\epsilon$ of
\begin{equation}
\label{phi}
\left\{
\begin{array}{rcl}
\epsilon^2 \varphi'' - 2\epsilon^2 K \varphi' - \kappa \varphi &=& 0 \\
\varphi(\pm \delta )&=&1/2\\
\end{array}
\right.
\end{equation}
satisfies: $ \varphi_\epsilon(0)<{e^{-\beta/\epsilon}}$ for all $\epsilon \in (0 , \epsilon_1)$.
\end{lem}
\begin{proof}
Note that 
$\varphi_\epsilon(\ti)=A e^{c_1\ti}+ B e^{c_2\ti},$
with $A,B$ given by

\[ 
A= \frac{1-e^{2c_2\delta}}{2(e^{c_1\delta}-e^{(2c_2-c_1)\delta})}, \hspace{0.5cm} \hbox{ and }\hspace{0.5cm} B= \frac{1-e^{2c_1\delta}}{2(e^{c_2\delta}-e^{(2c_1-c_2)\delta})}; 
\]
where $c_1, c_2$ are the roots of the equation $\epsilon^2 x^2 - 2\epsilon^2K x -\kappa =0$.
Then
\[ c_1,c_2(\epsilon)= K \pm \sqrt{\mu_\epsilon} \]
where $\mu_\epsilon= K^2 + \kappa/\epsilon^2$.
Now it is easy to see that $c_1(\epsilon)\to +\infty$, $c_2(\epsilon)\to -\infty$ as $\epsilon\to 0$ and consequently
\[
e^{\sqrt{\mu_\epsilon}\delta}(A+B)\to C, \hspace{0.3cm} \hbox{as }\epsilon\to 0
\]
where $C$ is some positive constant.
It follows that there are constants $\beta>0$ and $\epsilon_1>0$ such that
\[
\varphi_\epsilon(0)= A+B < e^{-\beta/\epsilon}, \hspace{0.3cm} \hbox{if }\epsilon<\epsilon_1.
\]
\end{proof}
\hfill$\square$
\vspace{0.3cm}

Now we shall study the behavior of the solutions of the equation 
 \begin{equation}
 \label{zzeta}
Z''(s) +Z(s)^5-Z(s)=0, \hspace{0.5cm}  Z'(0)=0\\
\end{equation}
when $s\to -\infty.$ To this end consider the following lemma.
\begin{lem}
Let $Z$ a  solution to the Eq. (\ref{zzeta}) such that 
\begin{eqnarray}
\label{zcondition1}
Z'(0)&=&0\\
Z(0)&=&\al
\end{eqnarray}
with $\al > 0$. Then
\begin{itemize}
 \item If $\al<3^{1/4}$ and $\al\neq 1$ then $Z$ oscillates around $1$.
 \item If $\al>3^{1/4}$ then $Z$ vanishes at some $s<0$, and it is positive and increasing in $(s,0)$.
 \item If $\al=3^{1/4}$ then $Z$ is  increasing in $(-\infty ,0)$ and
$
   \lim_{s\to -\infty} Z(s)=0.
$
\end{itemize}
\end{lem}

\begin{proof}
If we multiplicate the equation $(\ref{zzeta})$ by $Z'$ and integrate, then we have
\begin{equation}
\label{ctec}
 c=\frac{Z'(s)^2}{2}+\frac{Z(s)^6}{6}-\frac{Z(s)^2}{2}.
\end{equation}
It immediately follows that $Z$ is globally defined and 
$
c=\al^6/6-\al^2/2.
$
\par
\noindent
Note that if $s_1$ is a critical point of $Z$ then
\begin{equation}
\label{Z(s_1)}
  c=\frac{Z(s_1)^6}{6}-\frac{Z(s_1)^2}{2}.
\end{equation}
Now if $c\geq 0$, ie $\al \geq 3^{1/4}$, there is only one positive value of $Z(s_1)$ which satisfies the previous equation.
There are two options: either $Z$ vanishes  at some $s_0<0$ or $L=\lim_{s\to -\infty}Z(s)$ exists and it is non-negative.
Suppose first 
that $Z$ vanishes at some $s_0<0$ and that $Z'(s_0)\neq 0$ because the uniqueness of solutions.
Evaluating in $(\ref{ctec})$ we get 
$
Z'(s_0)^2/2=c
$
and $c>0$. Otherwise if there is a $L\ge0$ such that 
\begin{equation}
\label{limit}
 L=\lim_{s\to -\infty}Z(s).
\end{equation}
Then there is a sequence $s_j\to -\infty$ as $j\to \infty$ such that $Z'(s_j)\to 0$. 
If we take the limit when $s_j\to -\infty$ to the equation $(\ref{Z(s_1)})$, then 
we obtain $L=\al$, which is a contradiction because $Z$ is increasing, or $L=0$, which implies $c=0$.

Now if if $c<0$, ie $\al < 3^{1/4}$, and $Z$ has a critical point in $s_1$ there is two possible values of $Z(s_1)$: a minimum less than $1$ and a maximum greater than $1$. 
If $Z$ is not oscillating, it remains over or below the value $1$. We will show that it is not possible. 
Suppose $Z$ remains below $1$. Then $Z$ is convex and positive, so there is a $0< L<1$ that satisfies $(\ref{limit})$.
Morover $\lim_{j\to \infty}Z'(s_j)=0$ and $\lim_{j\to \infty}Z''(s_j)=0$ for a sequence $s_j\to -\infty$. 
Taking limit when $s_j\to-\infty$ in $(\ref{zzeta})$ we have: $\lim_{s_j\to -\infty}Z''(s_j)=L-L^5$. There is a contradiction.
If $Z$ is over the value $1$, we get a contradiction in a similar way. Therefore $Z$ remains oscillating around $1$.


\end{proof}
\hfill$\square$
\begin{lem}
\label{lemaZ}
Let $z=z_\epsilon$ be a solution of the equation
\begin{eqnarray}
\label{zeta}
 z''(s) + 2\epsilon \frac{ \cos(2T_0 + 2s \epsilon)}{\sin(2T_0 + 2s \epsilon)} z'(s) +z(s)^5-z(s)=0 
\end{eqnarray}
which is positive and increasing on the interval $(\psi(\epsilon),0)$ with the initial conditions 
\begin{eqnarray}
 z'(0)&=&0\\
z(0)&=&u_0(\epsilon)
\end{eqnarray}
and let $Z_0$ be the unique solution of problem (\ref{zzeta})-(\ref{zcondition1}) such that 
$
 Z_0 (0)=3^{1/4}.
$
Assume that 
$\psi(\ti)$ is a function such that $\psi(\epsilon)\to -\infty$ as $\epsilon\to 0$. 
Then
\[
 z_\epsilon(s) \to Z_0(s) \hspace{0.2cm}\hbox{ and }\hspace{0.2cm} z_\epsilon '(s) \to Z_0'(s)\hspace{0.2cm} 
 \hbox{ when } \epsilon \to 0
\]
 uniformly over bounded intervals
 and, in particular,

 \begin{equation}
 \label{u_0 to sigma}
    u_0(\epsilon)\to 3^{1/4} \hspace{0.2cm}\hbox{ when }\epsilon \to 0.
 \end{equation}

\end{lem}
\vspace{0.3cm}
\begin{proof} 
It is knwon that such  solutions $z_\epsilon$ are uniformly bounded (it can be proved for instance as in
\cite[Lemma 15]{Kwong}). 
Since the family of solutions $\{z_\epsilon(s): 0<\epsilon<\epsilon_0\}$ is equicontinuous it follows from Arzel\`a-Ascoli Theorem that
\[
z_\epsilon(s) \to Z(s)
\]
along a sequence, uniformly on bounded intervals, where $Z$ is a solution of (16). But on a large interval the solution $Z$ must be positive and 
increasing, therefore $Z=Z_0$ by the previous Lemma.
It then follows that the entire family converges to $Z_0$. In a similar manner it is proved that
$
z_\epsilon'(s) \to Z_0'(s).
$
\end{proof}
\hfill$\square$
\vspace{0.3cm}



\section{Proof of the main Theorems}
This section is devoted to the proofs of Theorems \ref{nonexis} and \ref{ksoluciones}.
The proof of Theorem \ref{nonexis} is based on the techniques used by C. Bandle and R. Benguria in \cite{BB} to prove a nonexistence result.
\par
\begin{demnon}
Multiply Eq. (\ref{toro}) by $u'(\theta)$ and integrate  over $(0,\theta_1)$. This yields
\begin{equation}\label{ne1}
\frac{1}{2}u'(\theta_1)^2+2 \displaystyle\int_{0}^{\theta_1} \frac{\cos(2\theta)}{\sin(2\theta)} (u'(\ti))^2 \, d\theta = -\lambda F(\alpha).
\end{equation}
If $ 0<\theta<\theta_1< \frac{\pi}{4}$ then $\frac{\cos(2\theta)}{\sin(2\theta)}>0$. 
Since $ \lambda <0 $, we have a contradiction:
\[
0<\frac{1}{2}u'(\theta_1)^2+2 \displaystyle\int_{0}^{\theta_1} \frac{\cos(2\theta)}{\sin(2\theta)} u'(\ti)^2 \, d\theta = -\lambda F(\alpha)<0.
\]
\end{demnon}
\hfill$\square$

Now we prove Theorem \ref{ksoluciones} for $k=1$. 
We shall show that there exist at least two solutions of problem (\ref{toro}) 
with initial value in the interval $(0,1)$ that have a single spike. 
Let 
$$
\epsilon^2=\displaystyle\frac{1}{|\lambda|}
$$
and $\alpha \in (0,1)$. Then consider the initial value problem

\begin{equation}\label{alpha}
\left\{
\begin{array}{rcl}
\epsilon^2 u''(\theta) + 2\epsilon^2 \frac{ \cos(2\theta)}{\sin(2\theta)} u'(\theta) + u(\theta)^5 - u(\theta)&=& 0  
\hspace{0.5cm}\hbox{in } (0,\theta_1)\\
 u&>&0  \hspace{0.5cm}\hbox{in } (0,\theta_1)\\
u(0)&=&\alpha\\
u'(0)&=&0.\\
\end{array}
\right.
\end{equation}
We denote the solution by $u_{\al, \epsilon}( \theta )$
and define 
\begin{equation}\label{teta-alfa-epsilon}
\Theta(\al, \epsilon)=\sup \{\ti \in (0,{\pi}/{2}): u_{\al, \epsilon} > 0 \hspace{0.5cm} \hbox{in } (0,\ti) \}.
\end{equation}
We will show that for $\epsilon$ small enough there are 
two values $\al_1, \al_2 \in (0,1)$ such that $\Theta(\al_i, \epsilon)=\ti_1$ for $i=1,2$ and the solutions $u_{\al_1, \epsilon}$ 
and $u_{\al_2, \epsilon}$ have exactly $1$ spike on the interval $(0,\ti_1).$ These techniques have been used successfully in \cite{BP}. 
\par
Note that Theorem \ref{nonexis} implies that  $\Theta(\al, \epsilon)> \frac{\pi}{4}$.
It may happen that the solution does not vanish in the interval $ (0,\frac{\pi}{2})$.
Therefore we define $\mathcal{A}(\epsilon)$ as the set of values of $\al$ for which $u_{\al, \epsilon}$ vanishes before $\frac{\pi}{2}$:
\begin{equation}\label{mathcalA}
\mathcal{A}(\epsilon)= \{ \al \in (0,1): 0<\Ti(\al, \epsilon)<{\pi}/{2}\}.
\end{equation}
$\mathcal{A} (\epsilon )$ is an open set and  if $\al \in \mathcal{A} (\epsilon )$ then it follows by uniqueness that 
\[ 
u_{\al, \epsilon}(\Ti(\al, \epsilon))=0  \hspace{0.5 cm}\hbox{ and }\hspace{0.5 cm} u'_{\al, \epsilon}(\Ti(\al, \epsilon))<0.
\]
On the other hand if we fix a $T_0 \in (\frac{\pi}{4},\frac{\pi}{2})$, then by the Sturm Comparison Theorem
for $\epsilon$ small enough the solution $w_\lambda$ of the linear Eq. (\ref{eqw}) has a maximum in $\tau_1^0(\lambda)<T_0$. 
Hence there exists an initial value $\al_0 \in (0,1) $ such that
\[  \tau_1(\al_0)=T_0. \] 
Since $\al_0 $ depends on $\epsilon$ denote
\[
 \al_0=\al_0(\epsilon);\hspace{0.3cm} u_\epsilon(\ti)=u_{\al_0(\epsilon), \epsilon}(\ti)\hspace{0.3cm}\hbox{and}\hspace{0.3cm} 
u_0(\epsilon)=u_\epsilon(T_0).
\]
In other words, fixed  $T_0 \in (\frac{\pi}{4},\frac{\pi}{2})$ and $\epsilon$ small enough,
we can find a solution $u_\epsilon$ of problem (\ref{alpha}) that reaches a first maximum $  u_0(\epsilon)$ at $\ti = T_0$.
It is clear that $  u_0(\epsilon) >1$.
In the following lemmas we show that for $\epsilon$ small enough, $F(u_0(\epsilon))> 0$, where $F$ is the function defined in $(\ref{F(u)})$:
$$F(u)=\frac{1}{6}u^6 - \frac{1}{2}u^2.$$
Then, since $F$ is increasing on $(1, \infty)$ and $ u_0(\epsilon)> 1$, it follows that $ u_0(\epsilon) >\sigma$, where
$\sigma=3^{1/4}$ is the positive zero of $F$. 

\begin{lem}\label{lemaA}
There exist constants $A>0$ and $\epsilon_0>0$ such that
\[F( u_0(\epsilon))>A \epsilon \hspace{0.5cm} \hbox{  for  } \epsilon<\epsilon_0. \]
\end{lem}

Consider the energy function $E_{\al_0}(\ti)$  associated with the solution $u_{\al_0}$ defined in (\ref{energy}) with $\epsilon^2=\displaystyle\frac{1}{|\lambda|}$. It satisfies
\begin{equation}\label{cota}
 E_{\al_0}(0)= \frac{1}{\epsilon^2} F(\al_0) 
\hspace{1cm}
\hbox{and}
\hspace{1cm}
  E_{\al_0}(T_0) = \frac{1}{\epsilon^2} F( u_0(\epsilon)).
\end{equation}
Integration of $E_{\al_0}'$ over $(0,T_0)$ yields
\[F( u_0(\epsilon))-F(\al_0)= -2 \epsilon^2 \displaystyle\int_0^{T_0} \frac{\cos(2\ti)}{\sin(2\ti)} u_{\epsilon}'(\ti)^2 \, d\ti.\]
Define
\begin{equation}\label{jotas}
\begin{array}{rcl}
J_1(\epsilon)&=&-2 \epsilon^2 \displaystyle\int_0^{\frac{\pi}{4}} \frac{\cos(2\ti)}{\sin(2\ti)} u_{\epsilon}'(\ti)^2 \, d\ti,  
\\
\\
J_2(\epsilon)&=&-2\epsilon^2  \displaystyle\int_{\frac{\pi}{4}}^{T_0} \frac{\cos(2\ti)}{\sin(2\ti)} u_{\epsilon}'(\ti)^2 \, d\ti.\\
\end{array}
\end{equation}
The expression for $F ( u_0(\epsilon))$ then becomes
\begin{equation}\label{F_0}
F(u_0(\epsilon))= F(\al_0) + J_1(\epsilon) + J_2(\epsilon).
\end{equation}
The following lemmas are used to estimate the terms on the right hand side of (\ref{F_0}).

Let  $\kappa>0$ be a constant such that
\begin{equation}\label{twei}
s^5 -s+\kappa s < 0   \hspace{0.5cm}  \hbox{for } 0<s<1/2.
\end{equation}
Write $\ti=T_0 + \epsilon s$ and let $z_{\epsilon}(s)=u_{\epsilon}(\ti)$. Then $z_\epsilon$ solves problem (\ref{zeta}).
By Lemma \ref{lemaZ} we know that if $Z_0$ is the solution of (\ref{zzeta})-(\ref{zcondition1}) such that $Z(0)=3^{1/4}$, then there is a $s_0<0$ such that 
 $Z_0(s_0)=1/4$ and hence $z_{\epsilon}(s_0)=u_\epsilon(T_0 + \epsilon s_0) \to 1/4$ as $\epsilon \to 0$.
It follows that for $\epsilon$ small enough,
\[u_\epsilon(T_0 + \epsilon s_0)<\frac{1}{2}.\]
Let $t_0= T_0 + \epsilon s_0$, with $\epsilon$ so that $\frac{\pi}{4}<t_0<T_0$ . 
Since $u$  is increasing on $(0,T_0)$, it yields
\begin{equation}\label{eins}
u_\epsilon(\ti)<\frac{1}{2}  \hspace{0.5cm} \hbox{and }   \hspace{0.5cm} u_{\epsilon}^5(\ti)-u_{\epsilon}(\ti)+\kappa u_{\epsilon}(\ti)<0
\end{equation}
for $0<\ti<t_0.$
\begin{lem}\label{u(t_1)}
Suppose  
$u_{\epsilon}$ is a solution of problem (\ref{alpha}) which is monotone on an 
interval $[t_1 - \delta, t_1 + \delta ] \subset (0, \pi /2 )$ and $u_{\epsilon}(t_1\pm \delta)<1/2$. Then there exists a constant $\beta>0$ and $\epsilon_1 >0$ such that if $\epsilon \in (0, \epsilon_1 )$ then
\[ u_{\epsilon}(t_1) \leq e^{-\frac{\beta}{\epsilon}}.\]
\end{lem}

\begin{proof}
Suppose that $u_{\epsilon}$ is  increasing on  $(t_1-\delta, t_1 + \delta)$ and choose $K$ such that 
$\frac{\cos(2\ti)}{\sin(2\ti)}+K <0$ for $\ti \in (t_1-\delta, t_1 + \delta)$.
Let $\varphi_\epsilon$ the solution of problem (\ref{phi}) centered in $t_1$.
Let $v=\varphi_\epsilon-u_{\epsilon}$. Thus the function $v$ satisfies
\begin{equation}
\begin{array}{rcl}
 \epsilon^2v'' - 2\epsilon^2 K v' - \kappa v &=& -\epsilon^2 u_{\epsilon}''+ 2\epsilon^2 K u_{\epsilon}'+\kappa u_{\epsilon} \\
                         &=& 2 \epsilon^2 (\frac{\cos(2\ti)}{\sin(2\ti)}+K ) u_{\epsilon}' +u_{\epsilon}^5-u_{\epsilon}+\kappa u_{\epsilon} \\
                         &<& 2 \epsilon^2 (\frac{\cos(2\ti)}{\sin(2\ti)} +K )u_{\epsilon}' \\
&\leq&0\\
\end{array}
\end{equation}
for $\ti \in (t_1-\delta, t_1 + \delta)$ because $u_{\epsilon}' \geq 0$.
Moreover $v(t_1\pm \delta )>0$. Then it follows from the Minimum Principle that $v(\ti)>0$ for all $\ti$ in the interval,
and in particular for $\ti=t_1$. It follows from Lemma \ref{lemaphi} that there exist $\beta , \epsilon_1>0 $ such that
if $\epsilon < \epsilon_1 $ 

\[
 u_{\epsilon}(t_1)< \varphi_\epsilon(t_1)<e^{-\beta/\epsilon}.
\]

The case when $u_{\epsilon}$ is  decreasing is proved similarly, picking $K$ such that $\frac{\cos(2\ti)}{\sin(2\ti)}+K >0$.
\end{proof}
\hfill$\square$


Then there exists an interval $(\pi/4-\delta, \pi/4+\delta)$ where the solution $u_\epsilon$ of problem (\ref{alpha}) is strictly 
increasing and so $u_\epsilon(\pi/4\pm\delta)<1/2.$ From Lemma \ref{u(t_1)} it follows that if $\epsilon<\epsilon_1$, then
\begin{equation}\label{previus}
u_{\epsilon}\left(\pi/4\right) \leq e^{-\frac{\beta}{\epsilon}}.
\end{equation}

\begin{lemanonum}[\bf{A}]
There exist positive constants $A$ and $\epsilon_0$ such that
\[ |J_1(\epsilon) |< A \epsilon^{-2} e^{-\frac{2\beta}{\epsilon}}  \hspace{0.5cm}  \hbox{  for } \epsilon <\epsilon_0.\]
\end{lemanonum}
\begin{proof}
Let $\ti <\frac{\pi}{4}$. Integration of Eq. (\ref{alpha}) over $(0,\ti)$ yields
\[
\epsilon^2 u_\epsilon '(\ti) = -2\displaystyle\int_0^\ti \frac{\cos(2s)}{\sin(2s)}u_\epsilon '(s)\, ds  +\displaystyle\int_0^\ti (u_\epsilon(s)-u_\epsilon(s)^5) \, ds. 
\]
Since $u>0$ on $(0,\pi/4)$ we have

\[
\epsilon^2 u_\epsilon '(\ti) < -2\displaystyle\int_0^\ti \frac{\cos(2s)}{\sin(2s)}u_\epsilon '(s)\, ds  +\displaystyle\int_0^\ti u_\epsilon(s)\, ds
\]
Note that $\frac{\cos(2s)}{\sin(2s)}>0$ for $ s\in (0,\ti)$ and $u_\epsilon$ is increasing on $(0,\pi/4)$.
Consequently
\[
 u_\epsilon '(\ti)<\epsilon^{-2}u_\epsilon(\pi/4)\ti 
\]
and by previous remark we have
\[
 u_\epsilon '(\ti)^2<\epsilon^{-4}u_\epsilon(\pi/4)^2\ti ^2 <\epsilon^{-4} e^{-2\beta/\epsilon} \ti ^2
\]
 for all $\epsilon<\epsilon_1$. Finally
\begin{equation}
\begin{array}{rcl}
|J_1(\epsilon)|&=&2\epsilon^2  \displaystyle\int_{0}^{\frac{\pi}{4}} \frac{\cos(2\ti)}{\sin(2\ti)} u_{\epsilon}'(\ti)^2 \, d\ti.\\
&<&2 \epsilon^{-2} e^{-2\beta/\epsilon}\displaystyle\int_0^{\frac{\pi}{4}} \frac{\cos(2\ti)}{\sin(2\ti)} \ti^2 \, d\ti.\\
\end{array}
\end{equation}
Let $ A:=2\displaystyle\int_0^{\frac{\pi}{4}} \frac{\cos(2\ti)}{\sin(2\ti)} \ti^2 \, d\ti >0$. Then we have
$| J_1(\epsilon) |< A \epsilon^{-2} e^{-2\beta/\epsilon} .$
\end{proof}
\hfill$\square$

\begin{lemanonum}[\bf{B}]
There exist constants $B$ and $\epsilon_0>0$ such that
\[J_2(\epsilon)\geq B \epsilon  \hspace{0.5cm}\hbox{ for  } \epsilon<\epsilon_0. \]
\end{lemanonum}
\begin{proof}
We shall see that \[B:=\liminf_{\epsilon\to 0} \frac{1}{\epsilon}J_2(\epsilon) >0.\]
Replacing $u_{\epsilon}$ by $z_\epsilon$ in (\ref{jotas}), we have
\[
 J_2(\epsilon)=-2\epsilon  \displaystyle\int_{(\frac{\pi}{4}-T_0)/\epsilon}^{0}  \frac{ \cos(2T_0 + 2s\epsilon)}{\sin(2T_0 + 
 2s\epsilon)}z_{\epsilon}'(s)^2 \, ds.
\]
Let $Z_0$ be the solution of problem (\ref{zzeta})-(\ref{zcondition1}) with $Z(0)=\sigma$.
It follows from Lemma \ref{lemaZ} that for any $L>0$:
\[
 \displaystyle\int_{-L}^{0}\frac{ \cos(2T_0 + 2s\epsilon)}{\sin(2T_0 + 2s\epsilon)}z_{\epsilon}'(s)^2 \, ds 
 \to \frac{ \cos(2T_0)}{\sin(2T_0)} \displaystyle\int_{-L}^{0}Z_0'(s)^2 \, ds 
\]
as $ \epsilon \to 0$. Note that if  $ \frac{\frac{\pi}{4}-T_0}{\epsilon} < -L<0$ then we have that
\[
\frac{1}{\epsilon} J_2(\epsilon) 
\geq -2\displaystyle\int_{-L}^{0}\frac{ \cos(2T_0 + 2s\epsilon)}{\sin(2T_0 + 2s\epsilon)}z_{\epsilon}'(s)^2 \, ds,
\]
Then 
\begin{equation}
\begin{array}{rcl}
B:=\liminf \frac{1}{\epsilon} J_2(\epsilon)&\geq& -2\liminf \displaystyle\int_{-L}^{0}\frac{ \cos(2T_0 + 2s\epsilon)}{\sin(2T_0 + 
2s\epsilon)}z_{\epsilon}'(s)^2 \, ds \\
\\
&=& -2\lim_{\epsilon \to 0} \displaystyle\int_{-L}^{0}\frac{ \cos(2T_0 + 2s\epsilon)}{\sin(2T_0 + 2s\epsilon)}z_{\epsilon}'(s)^2 \, ds \\
\\
&=& -2\frac{ \cos(2T_0)}{\sin(2T_0)}\displaystyle\int_{-L}^{0}Z_0'(s)^2 \, ds >0.\\
\end{array}
\end{equation}
\end{proof}
\hfill$\square$

\begin{lemanonum}[\bf{C}]
For $\epsilon$ small enough, there exists a positive constant $C$ such that
\[|F(\al_0(\epsilon))|\leq C e^{-\frac{2\beta}{\epsilon}. }\]
\end{lemanonum}

\begin{proof}
Since $u_\epsilon$  is increasing on  $(0, \pi/4)$ it follows from (\ref{previus}) that
\[
\al(\epsilon)<u_\epsilon(\pi/4) \leq e^{-\frac{\beta}{\epsilon}}  \hspace{0.5cm}  \hbox{ for }  \epsilon <\epsilon_1,
\]
and since $|F|$ is increasing on $(0,1)$ it results that for $\epsilon$ small enough

\[
|F(\al(\epsilon))|<|F(e^{-\frac{\beta}{\epsilon}} )| \leq \frac{\kappa}{2}e^{-\frac{2\beta}{\epsilon}}.
\]

\end{proof}
\hfill$\square$

From Lemmas (A), (B) and (C) it follows that

\[
F(u_0(\epsilon))= F(\al_0) + J_1(\epsilon) + J_2(\epsilon),
\]
with
\[
 |J_1(\epsilon) |< A \epsilon^{-2} e^{-2\beta/\epsilon}, \hspace{0.5cm}
  J_2(\epsilon)\geq B \epsilon , \hspace{0.5cm}
 |F(\al_0(\epsilon))|\leq C e^{-2\beta/\epsilon}
\]
for $\epsilon$ small enough.
From there follows Lemma \ref{lemaA}.
\par
\hfill$\square$

\begin{figure}[th]
   \centering
    \includegraphics[width=0.5\textwidth]{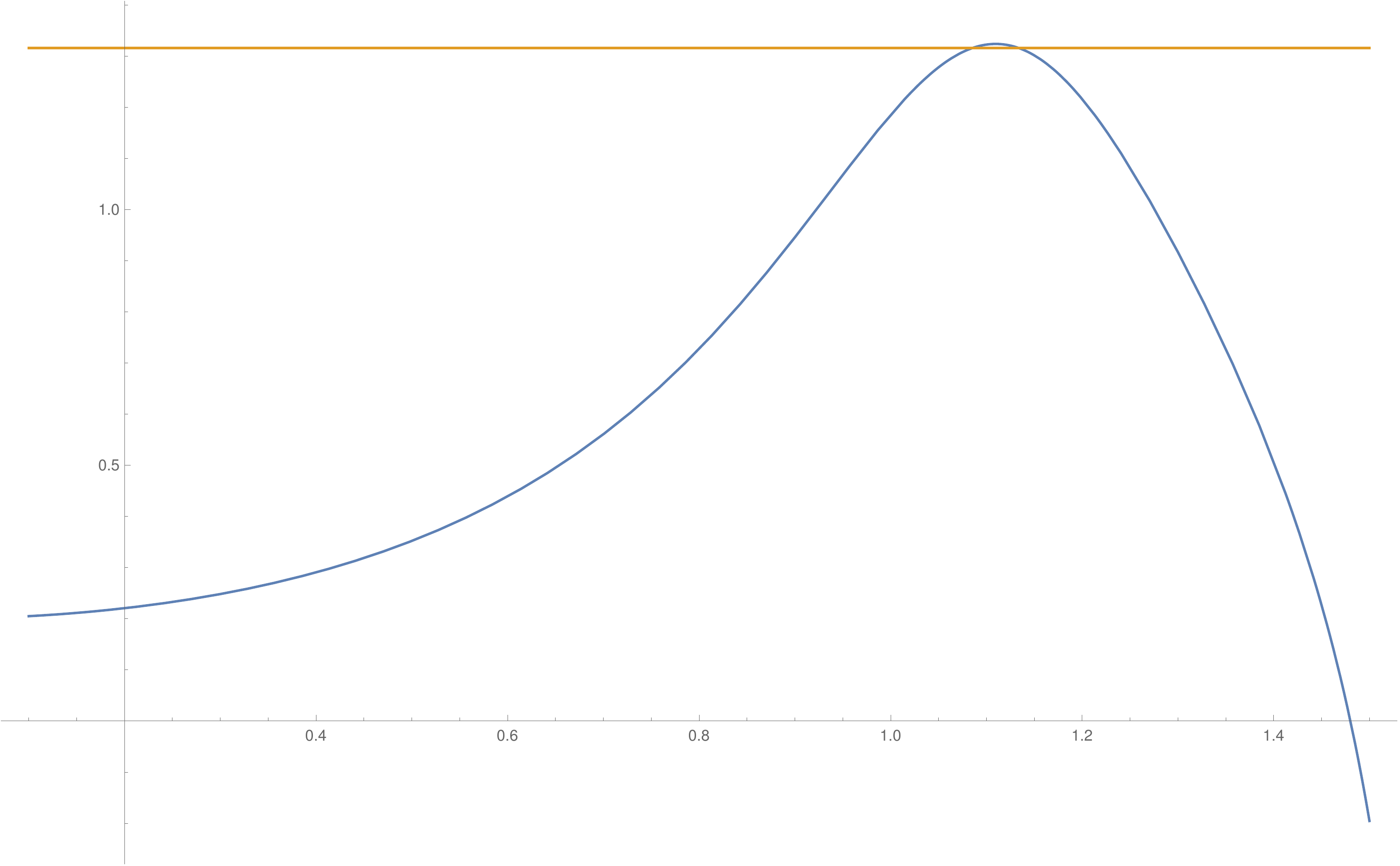}
        \vspace*{8pt}
    \caption{One-spike solution $u$ of problem (\ref{toro}) with $u_\epsilon (\tau_1)> \sigma$. }
\end{figure}
 
Fixed  $T_0 \in (\frac{\pi}{4},\frac{\pi}{2})$, we considered the solution $u_\epsilon$ of problem (\ref{alpha}) 
that reaches its first maximum at $\ti = T_0$ and we have proved that for $\epsilon$ is small enough $u_\epsilon(T_0)>\sigma$.
Next we show that the solution hits the $\ti$-axis shortly after $T_0$.


\begin{lem}\label{unmax}
There exists a constant $\epsilon_0>0$ such that for all $0<\epsilon<\epsilon_0$ there exists $\tau_\epsilon<\frac{\pi}{2}$ with 
the following properties: 
\[u_\epsilon(\tau_\epsilon)=0\] and
\begin{equation}
\left\{
\begin{array}{rcl}
u_\epsilon '(\ti)&>&0   \hspace{0.5cm}  \hbox{for } 0<\ti<T_0\\
u_\epsilon '(\ti)&<&0   \hspace{0.5cm}  \hbox{for }  T_0<\ti<\tau_\epsilon.\\
\end{array}
\right.
\end{equation}
Moreover
\begin{equation}\label{estimate}
|T_0 - \tau_\epsilon|= O(\sqrt{\epsilon})\hspace{0.5cm}\hbox{when } \epsilon \to 0.
\end{equation}
\end{lem}
\begin{proof}
Recall that $u_\epsilon$ has the following properties at $T_0$:
\[
u_\epsilon(T_0)>\sigma  \hspace{0.5cm}\hbox{and } \hspace{0.5cm} u_\epsilon '(T_0)=0.
\]
From Eq. (\ref{alpha}) it is easy to see that there is a constant $C>0$ such that
\begin{equation}\label{u''}
 u_\epsilon ''(\ti)<-\frac{C}{\epsilon^2}
\end{equation}
for all $\ti>T_0$ while $u_\epsilon (\ti)>\sigma$. Integration of (\ref{u''}) over $(T_0, \ti)$ yields
\[
 u_\epsilon '(\ti)<-\frac{C}{\epsilon^{2}}(\ti-T_0).
\]
We know that $|u_0(\ep)|<M$ for all $\epsilon$ small and for some $M>0$.
Then
\begin{equation}\label{ti-T0}
  u_\epsilon(\ti)-u_0(\ep)<-\frac{C}{2\epsilon^2}(\ti-T_0)^2.
\end{equation}
Since $u_0(\ep)>\sigma$ and $u$ is decreasing while $\ti>T_0$ and $u_\epsilon(\ti)>1$, 
there exists $\tau_\sigma>T_0$ such that 
\[
u_\epsilon(\tau_\sigma)=\sigma  \hspace{0.5cm} \hbox{ and }  \hspace{0.5cm} u_\epsilon '(\tau_\sigma)< 0.
\]
Taking $\ti=\tau_\sigma$ in (\ref{ti-T0}) it follows that 
\[
\sigma-u_0(\ep)<-\frac{C}{2\epsilon^2}(\tau_\sigma-T_0)^2
.\]
Finally we have
\begin{equation}\label{ineq1}
 |\tau_\sigma - T_0|=O(\epsilon).
\end{equation}

Next we use the energy function associated with $u_\epsilon$ defined in (\ref{energy}).
If $\ti \in  (\frac{\pi}{4},\frac{\pi}{2})$, then $E_\epsilon  '$ satisfies
\[E_\epsilon  '(\ti)=-2 \frac{\cos(2\ti)}{\sin(2\ti)}u_\epsilon '(\ti)^2 > 0.\]
Consequently integration of $E_\epsilon  '(\ti)$ over $(T_0,\tau_\sigma)$ yields
\[
0<\frac{u_\epsilon '(\tau_\sigma)^2}{2}-\frac{1}{\epsilon^2} F(u_0(\epsilon)).
\]
Therefore from Lemma \ref{lemaA} it follows:
\begin{equation}\label{fracA}
u_\epsilon '(\tau_\sigma)^2>\frac{2}{\epsilon^2} F(u_0(\epsilon))>\frac{A}{\epsilon}
\end{equation}
Define
\[ 
\tau_\epsilon= \sup \{T_0<\ti<\frac{\pi}{2} : u_\epsilon >0 \hbox{ and }u_\epsilon ' <0 \hbox{ on } (T_0 , \ti) \}.
 \] 
and integrate $E_\epsilon '$ over $(\tau_\sigma, \ti)$ with $\tau_\sigma<\ti<\tau_\epsilon$. Then
\begin{equation}\label{fracsqrtA}
\frac{u_\epsilon '(\ti)^2}{2} + \frac{1}{\epsilon^2} F(u_\epsilon (\ti)) > \frac{u_\epsilon  '(\tau_\sigma)^2}{2}.
\end{equation}
Since $F(u_\epsilon (\ti))<0$ and $u_\epsilon '(\tau_\sigma)<0$, it follows from (\ref{fracA}) and (\ref{fracsqrtA}) that
\[
 u_\epsilon  '(\ti)< u_\epsilon  '(\tau_\sigma)<-\sqrt{\frac{A}{\epsilon}}   \hspace{0.5cm}  \hbox{for } \tau_\sigma<\ti<\tau_\epsilon,
\]
Now we have:
\begin{equation}\label{ineq2}
 |\tau_\epsilon -\tau_\sigma|= O(\sqrt{\epsilon}).
\end{equation}
Write
\begin{equation}\label{triang}
|\tau_\epsilon - T_0| =  |\tau_\epsilon -\tau_\sigma| + |\tau_\sigma - T_0|. 
\end{equation}
Putting the estimates (\ref{ineq1})-(\ref{ineq2}) into (\ref{triang}) we obtain the estimate (\ref{estimate}).


\end{proof}
\hfill$\square$
\par
\noindent
This result allows us to establish the following
\begin{Propos}\label{prop1}
For $\epsilon$ small enough there exists $\al_0 \in \mathcal{A}(\epsilon)$  such that the solution  $u_{\al_0, \epsilon}$ 
of problem (\ref{alpha}) with initial value $\al_0(\epsilon)$ has exactly one spike.
\end{Propos}
\par
\noindent
Let $\mathcal{A}_1(\epsilon)$ the connected components of $\mathcal{A}(\epsilon)$  
such that the solutions $u_{\alpha , \epsilon}$ with $\alpha \in \mathcal{A}_1(\epsilon)$ have exactly
one spike. By the previous proposition for $\epsilon$ small enough $\mathcal{A}_1(\epsilon)$ is not
empty.

\begin{Propos}\label{prop2}  
Let $(\al_1^{-}, \al_1^{+}) \subset(0,1)$ be any connected component of $\mathcal{A}_1(\epsilon)$ and let
$\Ti(\al, \epsilon)$ be as in (\ref{teta-alfa-epsilon}). Then
\[\lim_{\al \to \al_1^{\pm}} \Ti(\al, \epsilon)= \frac{\pi}{2}.\]
\end{Propos}
\begin{proof}
Suppose that the assertion of Proposition \ref{prop2} is not true, so that
there exists a sequence $\{ \al_n \}$ which converges to, say $\al_1^{-}$, 
such that $\Ti(\al_n, \epsilon)$ converges to a point $\ti_\infty<\frac{\pi}{2}$. 
Then, by continuity $\Ti(\al_1^{-}, \ep)=\ti_\infty$ and therefore $\al_1^{-} \in \mathcal{A}(\epsilon)$, 
which contradicts the definition of $\al_1^{-}$.

\end{proof}
\hfill$\square$
\par
\noindent
This Proposition enables us to define:
\[
\Ti_{min,\epsilon}^1= \min \{ \Ti(\al, \epsilon): \al \in  \mathcal{A}_1(\epsilon)  \}.
\]
\begin{Propos}\label{prop3}
\begin{equation}\label{Ti_min}
\lim_{\epsilon \to 0} \Ti_{min,\epsilon}^1 =\frac{\pi}{4}.
\end{equation}
\end{Propos}
\begin{proof}
In order to prove Proposition \ref{prop1} we introduced an arbitrary
point $T_0>\frac{\pi}{4}$. We may choose this point arbitrarily close to $\frac{\pi}{4}$. 
In Lemma \ref{unmax} it has been shown that by choosing $\epsilon$ small enough,
we can achieve that $\tau_\epsilon$ is arbitrary close to $T_0$. Then we have (\ref{Ti_min}).
\end{proof}
\hfill$\square$
\par
\noindent

It follows from Proposition \ref{prop3} that, given $\ti_1\in (\frac{\pi}{4},\frac{\pi}{2})$, there exists $\epsilon_1>0$ such that 
if $\epsilon<\epsilon_1$, \[\frac{\pi}{4} <\Ti_{min,\epsilon}^1<\ti_1.\]
\par
\noindent

Let $\Gamma_1(\epsilon)=\{ (\al, \Ti(\al, \epsilon)): \al  \in (\al_1^{-},\al_1^{+})\} $, where $ (\al_1^{-},\al_1^{+})$
is a connected component of  $\mathcal{A}_1(\epsilon)$ such that 
$\min \{ \Ti(\al, \epsilon): \al \in  ( \al_1^{-},\al_1^{+})   \} < \theta_1 $.
Hence $\Gamma_1(\epsilon)$ intersects the line $\ti=\ti_1$ at least twice for all $\epsilon<\epsilon_1$. 
This yields at least two  $\al_1(\epsilon), \al_2(\epsilon) \in \mathcal{A}_1(\epsilon)$ such that 
$u_{\al_1(\epsilon)}, u_{ \al_2(\epsilon)}$ are solutions of problem (\ref{alpha}) 
having exactly one spike, and this completes the proof of Theorem \ref{ksoluciones} for the case $k=1$. In others words, we have proved that
for $\epsilon$ small enough there are at least two solutions with a single spike.

\par \vspace{0.5cm}
Now we prove Theorem \ref{ksoluciones} for $k=2$ in a similar way. 
We shall prove that
given any $\theta_1 > \pi/4$ there exists $\epsilon_2>0$ such that if $\epsilon<\epsilon_2$,then  problem (\ref{alpha}) 
has at least two solutions with initial value on $(0,1)$ that have exactly two spikes.

\begin{figure}[th]
    \centering
    \includegraphics[width=0.5\textwidth]{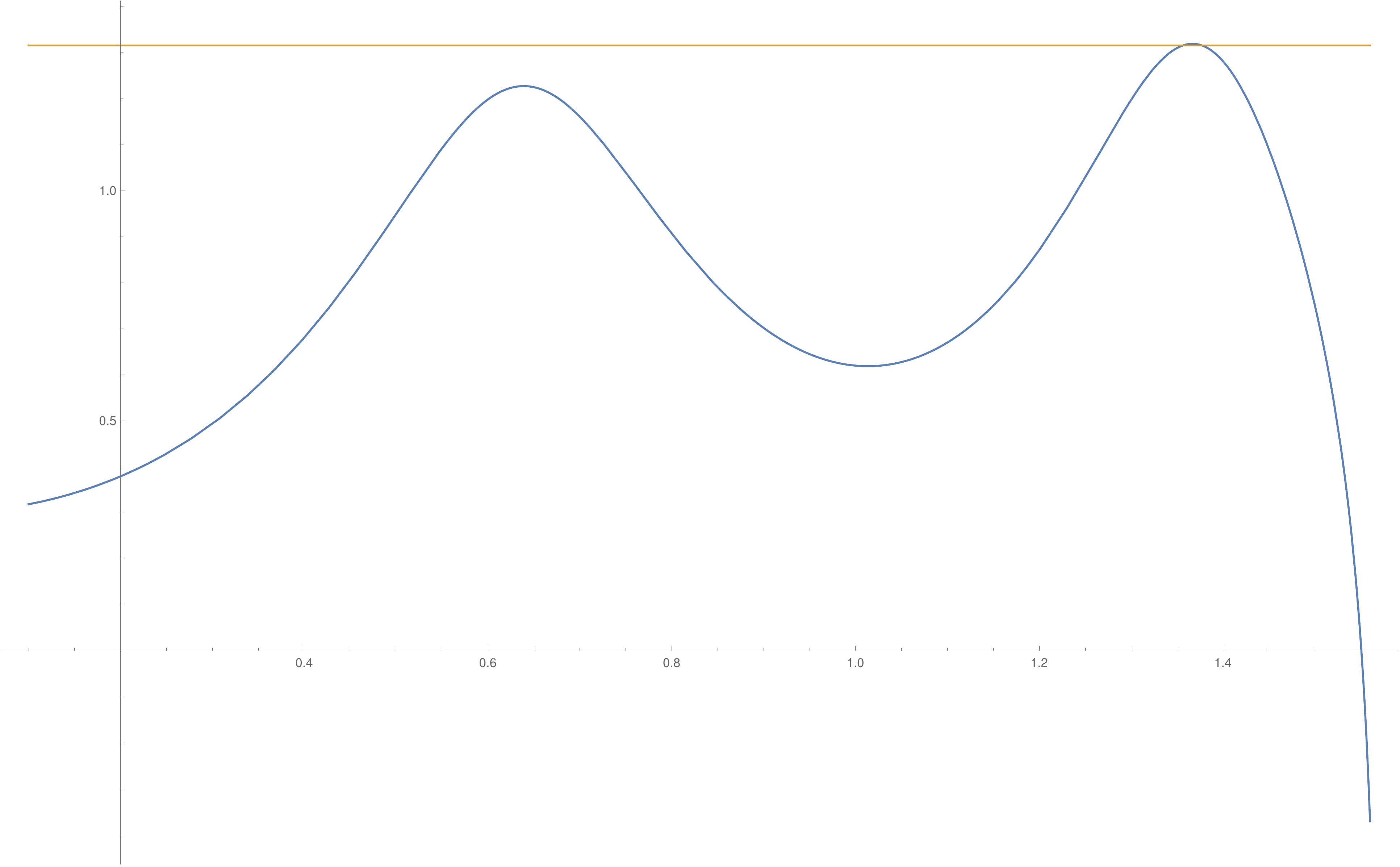}
        \vspace*{8pt}
    \caption{Two-spike solution $u$ of problem (\ref{toro}) with $u_0(\epsilon)\geq \sigma$}
\end{figure}

Repeating the argument we fix $T_0 \in (\frac{\pi}{4},\theta_1)$. For $\epsilon$ 
small enough we find an initial value
$\al_0 \in (0,1) $ such that
\[  
\tau_3(\al_0)=T_0. 
\] 
Write $\al_0=\al_0(\epsilon)$; $u_\epsilon(\ti)=u_{\al_0(\epsilon)}(\ti)$; 
$  u_0(\epsilon)=u_\epsilon(T_0)$ and $\tau_k(\al_0(\epsilon))=\tau_k(\epsilon)$.
Then 
\begin{equation}\label{taus}
\tau_1(\epsilon)<\tau_2(\epsilon)<\tau_3(\epsilon)=T_0.
\end{equation}
We have the following results:

\begin{lem}\label{tau+}
\begin{equation}\label{limsup}
 \limsup_{\epsilon\to 0}\tau_1(\epsilon)\leq \frac{\pi}{4}.
\end{equation}
\end{lem}


\begin{proof}
Let
$\tau_+=\limsup_{\epsilon\to 0}\tau_1(\epsilon)$.
  Note that $\tau_+\leq T_0$ and suppose that $\frac{\pi}{4} < \tau_+ \leq T_0 .$
Then, repeating the previous argument with $T_0$ replaced by $\tau_+$, we find that for $\epsilon$ small enough, the solution 
$u_\epsilon$ has a zero $\tau_\epsilon$ in a right neighbourhood of $\tau_+$ and is strictly decreasing on $(\tau_+, \tau_\epsilon)$. 
Since, by construction, $u_\epsilon$ has a local maximum at $T_0$ for every $\epsilon>0$, which lies above the
line $u = 1$, this is not possible. This completes the proof.
\end{proof}
\hfill$\square$


\begin{lem}\label{immediately}
Let $u_\epsilon$ be a $2$-spike solution of (\ref{alpha}) with $\tau_2(\epsilon)$ the second critical point. 
Then there are constants $\beta>0$ and $\epsilon_1>0$ such that
\[ u_{\epsilon}(\tau_2(\epsilon)) \leq e^{-\frac{\beta}{\epsilon}}  \hbox{  for  } \epsilon <\epsilon_1.\]
\end{lem}

\begin{proof}
\par
\noindent
From (\ref{taus}) and (\ref{limsup}) follows that for any 
fixed $\epsilon>0$ there exists $\delta >0$ (independent of $\epsilon$ small enough) such that
\begin{eqnarray}\label{either}
 \hbox{ either }\hspace{0.5cm} |\tau_2(\epsilon)- T_0|&>&2\delta \\\label{or}
 \hbox{ or }    \hspace{0.5cm} |\tau_2(\epsilon)- \pi/4|&>&2\delta.
\end{eqnarray}
\par
\noindent
Assume that we have  (\ref{either}) and let  $t_1$ be such that
\[
 (t_1- \delta, t_1+\delta)\subset (\tau_2(\epsilon), T_0) \hspace{0.5 cm}\hbox{and} \hspace{0.5 cm} u_\epsilon(t_1\pm \delta)<1/2.
\]
Then $u$ is increasing on $(t_1- \delta, t_1+\delta)$ and it follows from Lemma \ref{u(t_1)} that
\[
 u_{\epsilon}(\tau_2(\epsilon))< u_{\epsilon}(t_1)<e^{-\beta/\epsilon},
\]
where $\beta$ does not depend on $\epsilon$.
The case in which $\delta$ satisfies (\ref{or}) is analogous (note that the constant 1/2 used in Lemma 3.1 and Lemma 4.2 can be replaced for any other positive constant, 
as long as it is independent of $\epsilon$).
\end{proof}
\hfill$\square$

\begin{lem}
\[ 
\lim_{\epsilon\to 0}\frac{T_0-\tau_2(\epsilon)}{\epsilon}=\infty.
\]
\end{lem}

\begin{proof}
Note that $u_\epsilon$ is a positive solution of the equation
\begin{equation}\label{eqtau2}
 u_\epsilon ''(\theta) + 2 \frac{ \cos(2\theta)}{\sin(2\theta)} u_\epsilon '(\theta) =
 \frac{  u_\epsilon(\theta)-u_\epsilon(\theta)^5 }{\epsilon^2}
\end{equation}
such that
\[
  u_\epsilon(\tau_2(\epsilon))=e^{-\tilde{\beta}/\epsilon}\hspace{0.3cm}
   u_\epsilon '(\tau_2(\epsilon))=0\hspace{0.3cm}
    u_\epsilon(T_0)=u_0(\epsilon)\hspace{0.3cm}
     u_\epsilon '(T_0)=0
      \]
We want to show that $u_\epsilon(\tau_2(\epsilon)+\sqrt{\epsilon})<1$, because $u_\epsilon$ is 
increasing in the interval $(\tau_2(\epsilon), T_0)$ and $u_\epsilon(\tau_2(\epsilon))<1<u_\epsilon(T_0)$.
This means that $u_\epsilon$ cannot catch up $u_0(\epsilon)$ in the interval   $(\tau_2(\epsilon),\tau_2(\epsilon)+\sqrt{\epsilon})$. Then
\[
 \frac{T_0-\tau_2(\epsilon)}{\epsilon}>\frac{\sqrt{\epsilon}}{\epsilon}\to\infty
 \hspace{0.3cm} \hbox{ when } \epsilon\to 0.
\]
To see that, consider the linear auxiliary problem:
\begin{equation}
 w''(\theta) + 2 \frac{ \cos(2T_0)}{\sin(2T_0)} w'(\theta) =
 \frac{  w(\theta) }{\epsilon^2}
\end{equation}
with initial conditions
\[
  w(\tau_2(\epsilon))=e^{-\tilde{{\beta}}/\epsilon} \hspace{0.3cm} w'(\tau_2(\epsilon))=0.
      \]
Then by the Sturm Comparison Theory for all $0<\ti<T_0$, we have $u_\epsilon(\ti)<w(\ti)$.
In particular,
\[
u_\epsilon(\tau_2(\epsilon)+\sqrt{\epsilon})<w(\tau_2(\epsilon)+\sqrt{\epsilon}).
\]
Note that 
$w(\ti)=A e^{c_1(\ti-\tau_2(\epsilon))}+ B e^{c_2(\ti-\tau_2(\epsilon))},$
where $c_1, c_2$ are the roots of the equation 
$\epsilon^2 x^2 - 2\epsilon^2K x -1 =0$, with $K=-\frac{ \cos(2T_0)}{\sin(2T_0)}$.
Let $\mu_\epsilon= K^2 + \kappa/\epsilon^2$. Then $A$, $B$ are given by
\[
A=\frac{K-\sqrt{\epsilon}}{2\sqrt{\mu_\epsilon}}e^{-\tilde{\beta}/\epsilon}, \hspace{1cm} 
B=\frac{K+\sqrt{\epsilon}}{2\sqrt{\mu_\epsilon}}e^{-\tilde{\beta}/\epsilon}.
\]
Finally we have
\[
w(\tau_2(\epsilon)+ \sqrt{\epsilon})=\epsilon^{-\tilde{\beta}/\epsilon}+(K-\sqrt{\mu_\epsilon})\sqrt{\epsilon}+
\frac{K-\sqrt{\epsilon}}{2\sqrt{\mu_\epsilon}}e^{-\tilde{\beta}/\epsilon+2\sqrt{\mu_\epsilon}{\sqrt{\epsilon}}}.
 \]
Consequently, for $\epsilon$ small enough $w(\tau_2(\epsilon)+\sqrt{\epsilon})<1$.

\end{proof}
\hfill$\square$
\par
\noindent
Integration of $E_\epsilon '(\ti)$  over $(\tau_2(\epsilon),T_0)$ yields
\[
F(u_0(\epsilon)) - F(u_\epsilon(\tau_2(\epsilon))) =J(\epsilon)
\]
where \[J(\epsilon)= -2 \epsilon^2\displaystyle\int_{\tau_2(\epsilon)}^{T_0} \frac{\cos(2\ti)}{\sin(2\ti)} u_\epsilon'(\ti)^2 d\ti. \]
Then
\begin{equation}\label{Fuo}
 F(u_0(\epsilon)) = F(u_\epsilon(\tau_2(\epsilon)))+J(\epsilon) .
\end{equation}
Next we show that there is a constant $A> 0 $ such that $ F (u_0(\epsilon))> A\epsilon$ for $\epsilon$ enough small.

\begin{lemanonum}[\bf{$\tilde{B}$}]
There is a constant $C_1>0$ such that 
\[
J(\epsilon)>C_1\epsilon
\]
 for $\epsilon$ small enough.
\end{lemanonum}

\begin{proof}
To prove this lemma we may assume that $\tau_2(\epsilon)>\pi/4$, because when $\tau_2(\epsilon)<\pi/4$, the proof operates in the same way as before. 
Write $\ti=T_0 + \epsilon s$ and $z_{\epsilon}(s)=u_{\epsilon}(\ti)$ and replace $u$ by $z$ in $J$. Then, $z_\epsilon$ solves problem (\ref{zeta}) and
\[
 J(\epsilon)=
 -2\epsilon\displaystyle\int_{\frac{\tau_2(\epsilon)-T_0}{\epsilon}}^{0} \frac{\cos(2T_0+2s\epsilon)}{\sin(2T_0+2s\epsilon)}z_{\epsilon}'(s)^2 ds
\]
It follows from Lemma \ref{lemaZ} and Lemma 4.9 that for any $0<L<(T_0-\pi/4)/\epsilon$
\begin{equation}
\begin{array}{rcl}
C_1:=\liminf \frac{1}{\epsilon} J(\epsilon)&\geq& -2\lim_{\epsilon \to 0} 
\displaystyle\int_{-L}^{0}\frac{ \cos(2T_0 + 2s\epsilon)}{\sin(2T_0 + 2s\epsilon)}z_{\epsilon}'(s)^2 \, ds \\
 \\
&=& -2\frac{ \cos(2T_0)}{\sin(2T_0)}\displaystyle\int_{-L}^{0}Z_0'(s)^2 \, ds >0.\\
\end{array}
\end{equation}
\end{proof}
\hfill$\square$
\par
\noindent
The following lemma will be needed in order to complete the proof and follows immediately from Lemma \ref{immediately}.

\begin{lemanonum}[\bf{$\tilde{C}$}]
There is a constant $C_2>0$ such that 
\[
|F(u_\epsilon(\tau_2(\epsilon)))|<C_2e^{-\beta/2\epsilon}
\]
 for $\epsilon$ small enough.
\end{lemanonum}
\par
\vspace{0.3 cm}
\noindent
From Lemmas ($\tilde{B}$), ($\tilde{C}$) and the Eq. (\ref{Fuo}) we can see that $ F (u_0(\epsilon))> 0$
for $\epsilon$ enough small. Then it follows that $u_0(\epsilon)\geq \sigma$ and we can repeat the 
argument in Lemma 4.3 to prove  that $ u_ \epsilon $ has a zero $\tau_\ep \in(T_0, \frac{\pi}{2})$ such that $|T_0 - \tau_\epsilon|= O(\sqrt{\epsilon})$. 
It allows us to establish the following
\begin{Propos}\label{prop1'}
For $\epsilon$ small enough there exists $\al_0 \in \mathcal{A}(\epsilon)$  such that the solution  $u_{\al_0}(\ti)$ of problem  
(\ref{alpha})  with initial value $\al_0$ has exactly two spikes, where $\mathcal{A}(\epsilon)$  is the set defined in (\ref{mathcalA}).
\end{Propos}
Let $\mathcal{A}_2(\epsilon)$ be  the connected components 
of $\mathcal{A}(\epsilon)$ such that the solutions $u_{\alpha ,\epsilon}$ with $\alpha \in \mathcal{A}_2(\epsilon)$
 have exactly two spikes.
The proof of Theorem \ref{ksoluciones} for $k=2$ results from the following Propositions.

\begin{Propos}\label{prop2'}
Let $(\al_2^-,\al_2^+) \subset(0,1)$ be any connected component of $\mathcal{A}_2(\epsilon)$ and $\Ti(\al, \epsilon)$ as in (\ref{teta-alfa-epsilon}). 
Then
\[\lim_{\al \to \al_2^{\pm}} \Ti(\al, \epsilon)= \frac{\pi}{2}.\]
\end{Propos}
Now we can define:
\[
\Ti_{min,\epsilon}^2= \min \{ \Ti(\al, \epsilon): \al \in  \mathcal{A}_2(\epsilon)    \}.
\]

\begin{Propos}\label{prop3'}
\begin{equation}
\lim_{\epsilon \to 0} \Ti_{min,\epsilon}^2 =\frac{\pi}{4}.
\end{equation}
\end{Propos}

The it follows as in the case $k=1$ that there are  at least two  $\al_1(\epsilon), \al_2(\epsilon) \in \mathcal{A}_2(\epsilon)$ such that $u_{\al_1(\epsilon)}, u_{ \al_2(\epsilon)}$
are solutions of problem (\ref{alpha}) having exactly two spikes, and thus completes the proof of Theorem \ref{ksoluciones} for $k=2$.
\par
\vspace{0.3 cm}
\noindent
Finally, we turn to solutions with $k$ spikes. They are located at the points $\{\tau_{2j-1}:j= 1,2,\dots, k\}$. 
In the construction we fix $\tau_{2k-1}=T_0$ and we show that $\limsup \tau_{2(k-1)-1}\leq \pi/2$. Consequently $ F (u_0(\epsilon))> 0$ and
there exists $\al_0(\ep)$ such that the solution of (\ref{alpha}) $u_{\ep,\al_0(\ep)}$ has $k$ spikes.
This can be done with the methods developed in this section. 
Let $\mathcal{A}_k(\epsilon)$ be  the connected components 
of $\mathcal{A}(\epsilon)$ which contains the solutions with k spikes.
Let $(\al_k^-,\al_k^+) \subset(0,1)$ be any connected component of $\mathcal{A}_k(\epsilon)$. 
Then it can be shown that
\[\lim_{\al \to \al_k^{\pm}} \Ti(\al, \epsilon)= \frac{\pi}{2}.\]
Now we can define
$
\Ti_{min,\epsilon}^k= \min \{ \Ti(\al, \epsilon): \al \in   \mathcal{A}_k(\epsilon) \}
$
and it turns out that
$$\lim_{\epsilon \to 0} \Ti_{min,\epsilon}^k =\frac{\pi}{4}.$$
It follows that, given $\ti_1\in (\frac{\pi}{4},\frac{\pi}{2})$ we have that for $\epsilon$ small enough
 \[\frac{\pi}{4} <\Ti_{min,\epsilon}^k<\ti_1.\]

Then exactly as in the cases $k=1$ and $k=2$ we obtain at least two solutions of problem (\ref{alpha}) having exactly $k$ spikes.
This completes the proof of Theorem \ref{ksoluciones}.

\section*{Acknowledgments}
The author thanks Professors P. Amster and J. Petean for valuable observations and suggestions.



\begin{thebibliography}{88}


\bibitem{BN} 
{H. Brezis and L. Nirenberg}, Positive solutions of nonlinear elliptic equations involving critical Sobolev exponents, {\it Comm. Pure Appl. Math.} {\bf 36} (1983) 437--477.


\bibitem{KL} 
{M.K. Kwong and Y. Li},  Uniqueness of radial solutions of semilinear elliptic equations,{\it Trans. Amer. Math. Soc.} {\bf 333} (1992) 339--363.


\bibitem{BB} 
{C. Bandle and R. Benguria}, The Brezis-Nirenberg problem on $\mathbb{S}^3$, {\it J. Diff. Equ.} {\bf 178} (2002) 264--279.


\bibitem{BnP} 
{C. Bandle and L.A. Peletier}, Best constants and Emden equations for the critical exponent in $\mathbb{S}^3$, {\it Math. Ann.} {\bf 313} (1999) 83--93.


\bibitem{BPP} 
{H. Brezis, L.A. Peletier}, Elliptic equations with critical exponent on $\mathbb{S}^3$: new non-minimising solutions, {\it C. R. Acad. Sci. Paris}, Ser. I  {\bf 339} (2004).


\bibitem{BP} 
{H. Brezis and L. A. Peletier}, Elliptic equations with critical exponent on spherical caps of $\mathbb{S}^3$,{\it J. Anal. Math.} {\bf 98} (2006) 279--316.


\bibitem{P} 
{P. Padilla}, Symmetry properties of positive solutions of elliptic equations on symmetric domains,{\it Appl. Anal.} {\bf 64} (1997) 153-169.


\bibitem{KP} 
{S. Kumaresan and J. Prajapat}, Serrin’s result for hyperbolic space and sphere,{\it Duke Math. J.} {\bf 91} (1998) 17--28.


\bibitem{BnW} 
{C. Bandle and J. Wei}, Non-radial clustered spike solutions for semilinear elliptic problems on $\mathbb{S}^n$, {\it J. Anal. Math.} (2007) 102--181.


\bibitem{BnWei} 
{C. Bandle and J. Wei}, Multiple clustered layer solutions for semilinear elliptic problems on $\mathbb{S}^n$, {\it Comm. in Partial Differential Equations} {\bf 33}(4) 613--635.


\bibitem{J} 
{ J. Petean}, Metrics of constant scalar curvature conformal to Riemannian products,{\it Proc. Amer. Math. Soc.} {\bf 138} (2010) 2897--2905.


\bibitem{Kwong} 
{M.K. Kwong},  Uniqueness of positive solutions of $\Delta u -u =u^p =0$ in $\mathbb{R}^n$,{\it Arch. Rational Mech. Anal.} {\bf 105} (1989) 243--266.






\end{thebibliography}
\end{document}